\documentclass{amsart}
\usepackage{amsmath}
\usepackage{amsfonts}
\usepackage{amsthm}
\usepackage{amssymb}

\newtheorem{theorem}{Theorem}[section]
\newtheorem{lemma}[theorem]{Lemma}

\newtheorem{proposition}[theorem]{Proposition}

\theoremstyle{definition}

\newtheorem{Example}[theorem]{Example}

\theoremstyle{remark}

\numberwithin{equation}{section}

\begin{document}

\title[Riesz bases] {Riesz bases consisting of root functions of
1D Dirac operators}

\author{Plamen Djakov}

\thanks{P. Djakov acknowledges
the hospitality  of Department of Mathematics and the support of
Mathematical Research Institute of The Ohio State University, July -
August 2011.}

\address{Sabanci University,  Orhanli,
34956 Tuzla, Istanbul, Turkey}

\email{djakov@sabanciuniv.edu}

\author{Boris Mityagin}

\address{Department of Mathematics,
The Ohio State University,
 231 West 18th Ave,
Columbus, OH 43210, USA}

\email{mityagin.1@osu.edu}

\thanks{B. Mityagin acknowledges the support of the Scientific and
Technological Research Council of Turkey and the hospitality of
Sabanci University, April - June, 2011.}

\begin{abstract}
For one-dimensional Dirac operators
$$ Ly= i \begin{pmatrix}  1  &  0 \\ 0  & -1 \end{pmatrix}
\frac{dy}{dx} + v y, \quad v= \begin{pmatrix}  0  &  P \\ Q  & 0
\end{pmatrix}, \;\; y=\begin{pmatrix}  y_1 \\ y_2 \end{pmatrix},
$$
subject to periodic or antiperiodic boundary conditions, we give
necessary and sufficient conditions which guarantee that the system
of root functions contains  Riesz bases in $L^2 ([0,\pi],
\mathbb{C}^2).$

In particular, if the potential matrix $v$ is skew-symmetric (i.e.,
$\overline{Q} =-P$), or more generally if $\overline{Q} =t P$  for
some real $t \neq 0,$ then there  exists a Riesz basis that consists
of root functions of the operator $L.$ \vspace{2mm}

\end{abstract}

\maketitle

\section{Introduction}
We consider one-dimensional Dirac operators of the form
\begin{equation}
\label{001} L_{bc}(v) y= i \begin{pmatrix}  1  &  0 \\ 0  & -1
\end{pmatrix} \frac{dy}{dx} + v(x) \, y, \quad v= \begin{pmatrix}  0
&  P \\ Q  & 0
\end{pmatrix}, \;\; y=\begin{pmatrix}  y_1 \\ y_2 \end{pmatrix},
\end{equation}
with periodic matrix potentials $v$ such that $P,Q \in L^2 ([0,\pi],
\mathbb{C}^2),$ subject to periodic ($Per^+$) or antiperiodic
($Per^-$) boundary conditions ($bc$):
\begin{equation}
\label{002} Per^+: \;\;  y(\pi) = y(0); \qquad Per^-: \;\;  y(\pi) =
-y(0).
\end{equation}
Our goal is to give necessary and sufficient conditions on
potentials $v$  which guarantee that the system of periodic (or
antiperiodic) root functions of $L_{Per^\pm}(v)$ contains Riesz
bases.

The free operators $L_{Per^\pm}^0 = L_{Per^\pm}(0)$ have discrete
spectrum:
$$
Sp (L^0_{Per^\pm}) = \Gamma^\pm, \quad \text{where} \quad \Gamma^\pm
=
\begin{cases}
2\mathbb{Z} & \text{if} \; \;  bc=Per^+\\
2\mathbb{Z} +1& \text{if} \; \;  bc=Per^-\\
\end{cases}
$$
and each eigenvalue is of multiplicity 2. The spectra of perturbed
operators  $L_{Per^\pm}(v) = L^0_{Per^\pm} + v$ is also discrete;
for $n \in \Gamma^\pm $ with large enough $|n|$ the perturbed
operator has "twin" eigenvalues $\lambda_n^\pm $ close to $n.$ In
the case  where $\lambda^-_n \neq \lambda^+_n$ for large enough
$|n|,$ could the corresponding normalized "twin eigenfunctions"
form a Riesz basis?

Recently, in the case of Hill operators, many authors focused on
this problem  (see \cite{DV05,DM15,DM25a,DM25,GT11,Ma06-1,Ma06-2,
Ma06-3,ShVe09,Ve10} and the bibliography there). It may happen that
$\lambda_n^- \neq \lambda_n^+$ for $ |n|>N_*$ but the system of
 normalized eigenfunctions
fails to give a convergent eigenfunction expansion (see
\cite[Theorem 71]{DM15}).

In the present paper we consider such a problem in the case of 1D
periodic Dirac operators. In \cite{DM26},  we have singled out a
class of potentials $v$ which smoothness could be determined only by
the rate of decay of related spectral gaps $\gamma_n = \lambda_n^+ -
\lambda_n^-, $ where $ \lambda_n^\pm$ are the eigenvalues of
$L=L(v)$ considered on $[0,\pi]$ with periodic (for even $n$) or
antiperiodic (for odd $n$) boundary conditions. This class $X$ is
determined by the properties of the functionals $\beta_n^- (v;z) $
and  $\beta_n^+ (v,z) $ (see below (\ref{beta}) ) to be equivalent
in the following sense: there are $ c, \, N> 0$ such that
$$
c^{-1} |\beta_n^+ (v;z_n^*)| \leq |\beta_n^- (v;z_n^*)| \leq c
|\beta_n^+ (v;z_n^*)|, \quad  |n|>N,
 \quad
z_n^* = (\lambda_n^+ + \lambda_n^-)/2 -n.
$$

Section 3 contains the main results of this paper. We prove that if
$v \in X $ then the system of root functions of the operator
$L_{Per^\pm} (v) $ contains Riesz bases in $L^2 ([0,\pi],
\mathbb{C}^2). $ Theorem~\ref{thm0}, which is analogous to Theorem 1
in \cite{DM25} (or Theorem 2 in \cite{DM25a}), gives  necessary and
sufficient conditions for existence of such Riesz bases.
Theorem~\ref{thm00} is a modification of Theorem~\ref{thm0}  that is
more suitable for application to concrete classes of potentials.

Applications of Theorems \ref{thm0}  and \ref{thm00} are given in
Section 4. In particular, we prove that if the potential matrix $v$
is skew-symmetric (i.e., $\overline{Q} =-P$)   then the system of
root functions of $L_{Per^\pm} (v)$ contains Riesz bases in $L^2
([0,\pi], \mathbb{C}^2). $

\section{Preliminaries}

1. Let $H$ be a separable Hilbert space, and let $(e_\alpha, \,
\alpha \in \mathcal{I})$ be an orthonormal basis in $H.$ If $A: H\to
H$ is an automorphism, then the system
\begin{equation}
\label{p2} f_\alpha = A e_\alpha, \quad  \alpha \in \mathcal{I},
\end{equation}
is an unconditional basis in $H.$ Indeed, for each $x\in H$ we have
$$ x= A(A^{-1} x )= A \left (\sum_\alpha \langle A^{-1} x,e_\alpha
\rangle e_\alpha \right)= \sum_\alpha \langle x,(A^{-1})^*e_\alpha
\rangle f_\alpha =\sum_\alpha \langle x,\tilde{f}_\alpha \rangle
f_\alpha,$$ i.e., $(f_\alpha)$ is a basis, its biorthogonal system
is $\{\tilde{f}_\alpha =(A^{-1})^* e_\alpha, \;\alpha \in
\mathcal{I}\},$ and the series converge unconditionally.
 Moreover, it follows that
\begin{equation} \label{p3} 0< c \leq
\|f_\alpha\| \leq  C, \quad m^2\|x\|^2 \leq \sum_\alpha |\langle
x,\tilde{f}_\alpha \rangle|^2 \|f_\alpha\|^2 \leq M^2 \|x\|^2,
\end{equation}
with $ c= 1/\|A^{-1}\|, \; C=\|A\|,  \; M= \|A\|\cdot \|A^{-1}\|$
and $m=1/M.$

A basis of the form (\ref{p2}) is called {\em Riesz basis.}  One can
easily see that the property (\ref{p3}) characterizes Riesz bases,
i.e., a basis $(f_\alpha)$ is a Riesz bases if and only if
(\ref{p3}) holds with some constants $C\geq c>0$ and $M \geq m >0.$
Another characterization of Riesz bases  is given by the following
assertion (see \cite[Chapter 6, Section 5.3, Theorem 5.2]{GK}):
 {\em If
$(f_\alpha)$ is a normalized  basis (i.e., $\|f_\alpha\|=1 \;
\forall  \alpha $), then it is a Riesz basis if and only if it is
unconditional.}
\bigskip

A countable family of bounded projections $\{P_\alpha: H\to H , \,
\alpha \in \mathcal{I}\}$ is called {\em  unconditional basis of
projections} if $P_\alpha P_\beta =0 $ for $ \alpha \neq \beta $ and
$$ x= \sum_{\alpha \in \mathcal{I}} P_\alpha (x) \quad \forall x \in
H,
$$
where the series converge unconditionally in $H.$

If $\{H_\alpha, \, \alpha \in \mathcal{I}\} $ is a maximal family of
mutually orthogonal subspaces of $H$ and $Q_\alpha $ is the
orthogonal projection on $H_\alpha, \,  \alpha \in \mathcal{I},$
then   $\{Q_\alpha, \,  \alpha \in \mathcal{I}\}$ is an
unconditional basis of projections. A family of projections
$\{P_\alpha, \,  \alpha \in \mathcal{I}\}$ is
 called a {\em Riesz basis of projections} if
 there is a family of orthogonal projections
 $\{Q_\alpha, \,  \alpha \in \mathcal{I}\}$
 and an isomorphism $A:H \to H $ such that
\begin{equation}
\label{p5} P_\alpha = A \, Q_\alpha \, A^{-1}, \quad \alpha \in
\mathcal{I}.
\end{equation}
In view of (\ref{p5}),  if $\{P_\alpha \}$ is a Riesz basis of
projections, then there are constants $a, b>0 $ such that
\begin{equation}
\label{p6}       a \|x\|^2  \leq \sum_\alpha \|P_\alpha x\|^2 \leq b
\|x\|^2 \quad  \forall x \in H.
\end{equation}

For a family of projections
 $\mathcal{P}=\{P_\alpha, \,  \alpha \in \mathcal{I}\}$
 the following properties are equivalent (see \cite[Chapter 6]{GK}):

 (i)  $\mathcal{P}$  is an unconditional basis of projections;

 (ii)  $\mathcal{P}$  is a Riesz basis of projections.

\begin{lemma}
\label{lemr} Let $(P_\alpha, \, \alpha\in \mathcal{I})$  be a Riesz
basis of two-dimensional projections in a Hilbert space $H,$  and
let
  $f_\alpha, \, g_\alpha \in Ran \,
P_\alpha, $  $\alpha \in \mathcal{I}$   be linearly independent unit
vectors. Then the system $\{f_\alpha, \, g_\alpha, \; \alpha \in
\Gamma \} $ is a Riesz basis if and only if
\begin{equation}
\label{p7}  \kappa :=\sup |\langle  f_\alpha,  g_\alpha \rangle | <
1.
\end{equation}
\end{lemma}

\begin{proof}
Suppose that the system $\{f_\alpha, \, g_\alpha, \; \alpha \in
\mathcal{I} \} $ is a Riesz basis in $H.$ Then
$$
x= \sum_\alpha (f^*_\alpha (x) f_\alpha + g^*_\alpha (x) g_\alpha),
\quad x \in H,
$$
where $f^*_\alpha, g^*_\alpha$ are the conjugate functionals. By
(\ref{p3}), the one-dimensional projections
$$
P^1_\alpha (x) =f^*_\alpha (x) f_\alpha,  \quad P^2_\alpha
(x)=g^*_\alpha (x) g_\alpha, \quad \alpha \in \mathcal{I},
$$
are uniformly bounded. On the other hand, it is easy to see that
$$
\|P^1_\alpha\|^2 \geq \left ( 1- |\langle  f_\alpha,  g_\alpha
\rangle |^2 \right )^{-1}, \quad \|P^2_\alpha\|^2 \geq \left (1-
|\langle  f_\alpha,  g_\alpha \rangle |^2 \right )^{-1},$$ so
(\ref{p7}) holds.

Conversely, suppose (\ref{p7}) holds. Then we have for every $\alpha
\in \mathcal{I}$
$$
(1-\kappa) \left (  |f^*_\alpha (x)|^2 +|g^*_\alpha (x)|^2 \right )
\leq \|P_\alpha (x)\|^2 \leq (1+\kappa) \left (  |f^*_\alpha (x)|^2
+|g^*_\alpha (x)|^2 \right )
$$
which implies, in view of (\ref{p6}),
$$
\frac{a}{1+\kappa} \|x\|^2  \leq \sum_\alpha \left (  |f^*_\alpha
(x)|^2 +|g^*_\alpha (x)|^2 \right ) \leq \frac{b}{1-\kappa} \|x\|^2.
$$
Therefore, (\ref{p3}) holds, which means that the system
$\{f_\alpha, \, g_\alpha, \; \alpha \in \mathcal{I} \} $ is a Riesz
basis in $H.$
\end{proof}
\bigskip

2.  We consider the Dirac operator (\ref{001}) with $bc = Per^\pm$
in the domain
$$
Dom\, \left ( L_{Per^\pm}(v) \right ) = \left \{y= \begin{pmatrix}
y_1 \\y_2
\end{pmatrix}: \; y_1, y_2 \; \text{are absolutely continuous,} \;
y(\pi) = \pm y(0) \right \}.
$$
Then the operator $L_{Per^\pm}(v)$ is densely defined and closed;
its adjoint operator is
\begin{equation}
\label{adj} \left (   L_{Per^\pm}(v)   \right )^* = L_{Per^\pm}
(v^*), \quad
 v^* =\begin{pmatrix} 0 & \overline{Q}
\\ \overline{P} & 0 \end{pmatrix} .
\end{equation}

\begin{lemma} \label{loc}
 The spectra of the operators
$L_{Per^\pm} (v)$ are discrete. There is an $N=N(v)$  such that the
union $\cup_{|n|>N} D_n $ of the discs $D_n =\{z: \, |z-n|< 1/4 \}$
contains all but finitely many of the eigenvalues of $L_{Per^+}$ and
$L_{Per^-}$ while the remaining finitely many eigenvalues are  in
the rectangle $R_N = \{z: \; |Re \,z|, \, |Im \, z| \leq N+1/2 \}.$

Moreover, for $|n|>N$ the disc $D_n$ contains  two (counted with
algebraic multiplicity) periodic (if $n$ is even) or antiperiodic
(if $n$ is odd) eigenvalues $\lambda_n^-, \lambda_n^+$ such that $Re
\,\lambda_n^- < Re \,\lambda_n^+$  or $Re \,\lambda_n^- = Re
\,\lambda_n^+$ and $Im \,\lambda_n^- \leq Im \,\lambda_n^+.$
\end{lemma}

See details and more general results about localization of these
spectra in \cite{Mit03,Mit04} and \cite[Section 1.6]{DM15}.

  Lemma \ref{loc} allows
us to apply the Lyapunov--Schmidt projection method and reduce the
eigenvalue equation $Ly = \lambda y $ for  $\lambda \in D_n $ to an
eigenvalue equation in the two-dimensional space $E_n^0 = \{ L^0 Y =
n Y\}$   (see \cite[Section 2.4]{DM15}). This leads to the following
(see in \cite{DM15} the formulas (2.59)--(2.80) and Lemma~30).

\begin{lemma}
\label{lem1}

 (a) For large enough $|n|, \; n\in \mathbb{Z},$ there
are functionals $\alpha_n (v,z) $ and $ \beta^\pm_n (v;z), \; |z| <
1 $ such that a number $\lambda = n + z, \;|z| < 1/4, $ is a
periodic (for even $n$) or antiperiodic (for odd $n$) eigenvalue of
$L$ if and only if $z$  is an eigenvalue of the matrix
\begin{equation}
\label{p1}  \left [
\begin{array}{cc} \alpha_n (v,z)  & \beta^-_n (v;z)
\\ \beta^+_n (v;z) &  \alpha_n (v,z) \end{array}
\right ].
\end{equation}

(b) A number $\lambda = n + z^*, \;|z^*| < \frac{1}{4}, $ is a
periodic (for even $n$) or antiperiodic (for odd $n$) eigenvalue of
$L$ of geometric multiplicity 2  if and only if $z^*$  is an
eigenvalue of the matrix (\ref{p1})  of geometric multiplicity 2.
\end{lemma}

The functionals $\alpha_n (z;v) $ and $\beta^\pm_n (z;v)$ are well
defined for large enough $|n|$ by explicit expressions in terms of
the Fourier coefficients $p(m), \, q(m), \, m \in 2\mathbb{Z} $ of
the potential entries $P, Q$  about the system $\{e^{imx}, \, m \in
2\mathbb{Z}\}$ (see
 \cite[Formulas (2.59)--(2.80)]{DM15}).
Here we provide formulas only for $\beta^\pm_n (v;z):$
\begin{equation}
\label{beta} \beta_n^\pm (v;z) = \sum_{\nu=0}^\infty  \sigma^\pm_\nu
\quad \text{with}  \quad  \sigma^+_0= q(2n), \quad  \sigma^-_0=
p(-2n),
\end{equation}
$$\sigma^+_\nu =
\sum_{j_1 , \ldots, j_{2\nu} \neq n } \frac{q(n+j_1 ) p(-j_1 - j_2 )
q(j_2 +j_3 ) \ldots p(-j_{2\nu -1}- j_{2\nu} ) q(j_{2\nu} + n )} {
(n-j_1  +z) (n-j_2  +z) \ldots (n-j_{2\nu -1}  +z) (n-j_{2\nu}
+z)}, $$
$$\sigma^-_\nu= \sum_{j_1 , \ldots, j_{2\nu} \neq n } \frac{p(-n-j_1 ) q(j_1 +
j_2 ) p(-j_2 -j_3 )  \ldots q(j_{2\nu -1}+ j_{2\nu} ) p(-j_{2\nu} -
n )} {(n-j_1  +z) (n-j_2  +z) \ldots  (n-j_{2\nu -1} +z) (n-j_{2\nu}
+z)}, $$ where  $j_1, \ldots, j_{2\nu} \in n+ 2 \mathbb{Z}.$

Next we summarize some basic properties of $\alpha_n (z;v) $ and
$\beta^\pm_n (z;v).$

\begin{proposition}
\label{bprop} (a) The functions $\alpha_n (z;v) $ and $\beta^\pm_n
(z;v) $ depend analytically on $z$ for $|z|\leq 1.$  For $|n| \geq
n_0$  the following estimates hold:
\begin{equation}
\label{1.36} |\alpha_n (v;z)|, \, |\beta^\pm_n (v;z)| \leq C \left (
\mathcal{E}_{|n|} (r) +1/\sqrt{|n|} \right ), \quad  |z| \leq 1/2;
\end{equation}
\begin{equation}
\label{1.37} \left |\frac{\partial\alpha_n}{\partial z} (v;z)
\right|, \, \left |\frac{\partial\beta^\pm_n}{\partial z} (v;z)
\right |\leq C \left ( \mathcal{E}_{|n|} (r) +1/\sqrt{|n|} \right ),
\quad |z| \leq 1/4,
\end{equation}
where $ r= (r(m)), \; r(m) = \max \{|p (\pm m)|, q(\pm m) \}, \; C =
C(\|r\|), \; n_0 =n_0 (r)$  and
$$
\left (\mathcal{E}_m (r) \right )^2 = \sum_{|k| \geq m} |r(k)|^2 .
$$

(b) For large enough $|n|,$ the number $\lambda = n+ z, $ $ z\in
D=\{\zeta: |\zeta| \leq 1/4\}, $ is an eigenvalue of $L_{Per^\pm} $
if and only if
 $z\in D $ satisfies the basic equation
\begin{equation}
\label{be} (z-\alpha_n (z;v))^2 = \beta^+_n (z;v) \beta^-_n (z,v),
\end{equation}

(c) For large enough $|n|,$ the equation (\ref{be}) has exactly two
roots in $D$ counted with multiplicity.
\end{proposition}

\begin{proof}
The assertion (a) is proved in \cite[Proposition 35]{DM15}.
Lemma~\ref{lem1} implies  (b).  By (\ref{1.36}), $\sup_D |\alpha_n
(z)| \to 0 $ and $\sup_D |\beta^\pm_n (z)| \to 0 $ as $n\to \infty.
$ Therefore,
 (c) follows from the Rouch\'e theorem.
\end{proof}

In view of Lemma \ref{loc}, for large enough $|n|$  the numbers
$z^*_n = (\lambda^+_n + \lambda_n^-)/2 - n $ are well defined. The
following estimate
 of $\gamma_n $ from above
follows from (\ref{1.36})  and (\ref{1.37}) (see \cite[Lemma
40]{DM15}).

\begin{lemma}
\label{lem11}  For large enough $|n|,$
\begin{equation}
\label{1.42} \gamma_n =|\lambda^+_n - \lambda_n^-| \leq (1+\delta_n)
(|\beta_n^- (z^*_n) |+|\beta^+_n (z^*_n)|)
\end{equation}
with $\delta_n \to 0 $ as $|n| \to \infty. $
\end{lemma}

{\em Remark.}  Here and sometimes thereafter, we suppress the
dependence on $v$ in the notations and write $\alpha_n (z)$ and
$\beta_n^\pm (z). $ \bigskip

 3.  In view of the above consideration, there is
 $n_0=n_0 (v)$ such that
 $\lambda_n^\pm, \, \beta^\pm_n (z) $
 and $\alpha_n (z) $ are well-defined for $|n|>n_0, $
 and Lemmas \ref{loc},  \ref{lem1}, \ref{lem11} and
 Proposition~\ref{bprop} hold.
 Let us set
 \begin{equation}
\label{g101} \mathcal{M}^\pm= \{n \in \Gamma^\pm: \quad n\in
\Gamma^\pm, \;\;  |n|>n_0, \;\;
 \lambda^-_n \neq  \lambda^+_n
 \}.
\end{equation}

 {\em Definition. Let $X^\pm$
 be the class of all Dirac potentials $ v $
 with the following property:
 there are constants $c\geq 1 $
and $N\geq n_0 $ such that }
 \begin{equation}
\label{g1}  \frac{1}{c} |\beta^+_n (v;z^*_n)| \leq |\beta^-_n
(v;z^*_n)| \leq  c \,|\beta^+_n (v;z^*_n)| \quad \text{if} \;\;
     n\in \mathcal{M}^\pm, \;   |n| \geq N.
\end{equation}

\begin{lemma}
\label{lem21}  If $v \in X^\pm$ and the set $\mathcal{M}^\pm$ is
infinite, then for  $n \in \mathcal{M}^\pm$ with  sufficiently large
$|n|$ we have
\begin{equation}
\label{g2} \frac{1}{2}|\beta^\pm_n (v;z_n^*)|  \leq   |\beta^\pm_n
(v;z)| \leq 2 |\beta^\pm_n (v;z_n^*)| \quad \forall \,z \in K_n: =
\{z: |z-z_n^*|\leq \gamma_n \}.
\end{equation}
\end{lemma}

\begin{proof}
By Lemma \ref{lem11}, if $v \in X^\pm$ then for $n \in
\mathcal{M}^\pm$ with large enough $|n|$ we have $\beta^\pm_n
(z_n^*) \neq 0.$ In view of (\ref{1.37}),  if $z \in K_n $ then for
large enough $|n|$
$$
\left | \beta^\pm_n (z)- \beta^\pm_n (z_n^*) \right | \leq
\varepsilon_n  \left |z - z_n^* \right | \leq \varepsilon_n \,
\gamma_n,
$$
where $\varepsilon_n=C \left ( \mathcal{E}_{|n|} (r) +1/\sqrt{|n|}
\right ) \to 0 $  as $|n| \to \infty.$ By Lemma~\ref{lem11}, for
large enough $|n|$ we have $ \gamma_n \leq 2  \left ( |\beta^-_n
(z_n^*)| +|\beta^+_n (z_n^*)| \right ). $ Then, for $n\in
\mathcal{M},$
$$
\left | \beta^\pm_n (z)- \beta^\pm_n (z_n^*) \right | \leq
2\varepsilon_n  \left ( |\beta^-_n (z_n^*)| +|\beta^+_n (z_n^*)|
\right ) \leq 2\varepsilon_n (1+c) \left | \beta^\pm_n (z_n^*)
\right |
$$
which implies, for sufficiently large $|n|,$
$$
[1-2\varepsilon_n (1+c)]\left | \beta^\pm_n (z_n^*)  \right | \leq
\left | \beta^\pm_n (z)  \right | \leq [1+2\varepsilon_n (1+c)]\left
| \beta^\pm_n (z_n^*)  \right |.
$$
Since $\varepsilon_n \to 0$ as $|n| \to \infty, $ (\ref{g2})
follows.
\end{proof}

\begin{proposition}
\label{prop1} Suppose that $v \in X^\pm $ and the corresponding set
$\mathcal{M}^\pm$ is infinite. Then for $n\in \mathcal{M}^\pm$ with
large enough $|n|$
\begin{equation}
\label{g4} \frac{2\sqrt{c}}{1+4c}\,\left (|\beta^-_n
(v;z^*_n)|+|\beta^+_n (v;z^*_n)|\right ) \leq \gamma_n \leq 2 \left
(|\beta^-_n (v;z^*_n)|+|\beta^+_n (v;z^*_n)|\right ).
\end{equation}
\end{proposition}

\begin{proof} The estimate of $\gamma_n $ from above
follows from Lemma~\ref{lem11}. By Lemma \ref{lem11},  for $n \in
\mathcal{M}^\pm$ with large enough $|n|$ we have $\beta^\pm_n
(z_n^*) \neq 0.$ Set
 $$
 t_n = |\beta_n^+ (z_n^+)| /|\beta_n^- (z_n^+)|, \quad z_n^+ =
 \lambda_n^+ - n, \quad n \in
\mathcal{M}^\pm.$$ By Lemma~\ref{lem21},  $t_n $ is well defined for
large enough $|n|.$ By Lemma~49 in \cite{DM15}, there exists a
sequence $(\delta_n)_{n \in \mathbb{Z}}$ with $\delta_n \to 0 $  as
$|n| \to \infty $ such that, for   $n \in  \mathcal{M}^\pm$  with
large enough $|n|,$
\begin{equation}
\label{g5}  |\gamma_n | \geq \left (\frac{2\sqrt{t_n}}{1+t_n}
-\delta_n \right )  \left ( |\beta^-_n (z_n^*)| +|\beta^+_n (z_n^*)|
\right ).
\end{equation}

In view of (\ref{g2}) in Lemma \ref{lem21},  for large enough $|n|$
we have $    1/(4c)  \leq t_n  \leq 4c. $ Therefore, by (\ref{g5})
it follows
$$
 \gamma_n  \geq \left ( \frac{2\sqrt{4c}}{1+4c}-\delta_n \right )
  \left ( |\beta^-_n (z_n^*)| +|\beta^+_n (z_n^*)|
\right ),
$$
which implies (since $\delta_n \to 0 $ as $|n| \to \infty $) the
left inequality in (\ref{g4}). This completes the proof.
\end{proof}

\section{Riesz bases of root functions}

In view of Lemma~\ref{loc}, the Dirac operators $L_{Per^\pm}(v)$
have discrete spectra; for $N$ large enough and $n \in \Gamma^\pm $
with $|n|>N$ the Riesz projections
\begin{equation}
\label{r11} S^{\pm}_N = \frac{1}{2\pi i} \int_{ \partial R_N }
(z-L_{Per^\pm})^{-1} dz, \quad P^\pm_n = \frac{1}{2\pi i} \int_{
|z-n|= \frac{1}{4} } (z-L_{Per^\pm})^{-1} dz
\end{equation}
are well--defined and $\dim \,S^{\pm}_N <\infty,  \; \dim P^\pm_n =
2.$ Further we suppress in the notations the dependence on the
boundary conditions $Per^\pm$ and write $S_N, \, P_n $ only. By
\cite[Theorem 3]{DM20},
\begin{equation}
\label{r12} \sum_{n\in \Gamma^\pm,|n| > N} \|P_n - P_n^0\|^2 <
\infty,
\end{equation}
where $P_n^0$ are the Riesz projections of the free operator.
Moreover, the Bari--Markus criterion implies (see Theorem 9 in
\cite{DM20}) that the spectral Riesz decompositions
\begin{equation}
\label{r13}
 f = S_N f +  \sum_{n\in \Gamma^\pm,|n| > N}
  P_n f \qquad  \forall
 f \in L^2 \left ([0,\pi], \mathbb{C}^2 \right )
\end{equation}
converge unconditionally.  In other words, $\{S_N, \; P_n, \; n \in
\Gamma^\pm, \; |n|>N\} $ is a Riesz  basis of projections  in the
space $L^2 \left ([0,\pi], \mathbb{C}^2 \right ).$

\begin{theorem}
\label{thm0} (A)  If $v \in X^\pm,$ then there exists a Riesz basis
in $L^2 ([0,\pi],\mathbb{C}^2)$ which consists of root functions of
the operator $L_{Per^\pm}(v).$

(B)  If $v \not \in X^\pm, $ then the system of root functions of
the operator $L_{Per^\pm}(v)$ does not contain Riesz bases.
\end{theorem}

{\em Remark.}  To avoid any confusion, let us emphasize that in
Theorem~\ref{thm0} two {\em independent} theorems are stacked
together: one for the case of periodic boundary conditions $Per^+,$
and another one for the case of antiperiodic  boundary conditions
$Per^-.$

\begin{proof}
We consider only the case of periodic boundary conditions $bc=
Per^+$ since the proof is the same in the case of
 antiperiodic boundary conditions $bc=Per^-. $

(A)  Fix $v \in X^+, $   and let $N=N(v)>n_0 (v) $ be chosen so
large that Lemma~\ref{lem21},  Proposition~\ref{prop1} and
(\ref{r11})--(\ref{r13}) holds for $|n|>N.$

If $n \not \in \mathcal{M}^+ $   then
 $\; \lambda_n^*= n+z_n^*$  is a double eigenvalue.
 In this case we choose  $f(n),
g(n)\in Ran (P_n) $ so that
\begin{equation} \label{a2}
\|f(n)\|=\|g(n)\|=1, \quad L_{Per^+}(v) f(n) =\lambda_n^* f(n),
\quad
 \langle f(n), g(n) \rangle =0.
\end{equation}

If $n  \in \mathcal{M}^+ $   then $\lambda_n^- $ and $ \lambda_n^+$
are simple eigenvalues.  Now we choose corresponding eigenvectors
$f(n), g(n)\in Ran (P_n) $ so that
\begin{equation} \label{a2a}
\|f(n)\|=\|g(n)\|=1, \; \; L_{Per^+}(v) f(n) =\lambda_n^+ f(n), \;\;
L_{Per^+}(v) g(n) =\lambda_n^- g(n).
\end{equation}

Let $H $  be  the closed linear span of the system
$$\Phi =\{f(n), \, g(n): \; n \in \Gamma^+, \,\; |n|>N  \}.$$
By (\ref{r13}), $\; L^2 ([0,\pi], \mathbb{C}^2) = H\oplus Ran
(S_N).$ Since $\dim \, S_N < \infty,$
 the theorem will be proved if we show
that the system $\Phi $
 is a Riesz
basis in the space $H.$

By (\ref{r13}), the system of two-dimensional projections $\{P_n: \;
n \in \Gamma^+, \; |n|>N  \}$ is Riesz basis of projections in $H.$
By  Lemma~\ref{lemr},  the system $\Phi $ is a Riesz basis in $H$ if
and only if
$$
\sup_{n\in \Gamma^+, |n|>N} |\langle f(n), g(n) \rangle  | < 1.
$$
By (\ref{a2}), we need to consider only indices $n\in
\mathcal{M}^+.$
 Next we show that
\begin{equation} \label{a3}
\sup_{\mathcal{M}^+} |\langle f(n), g(n) \rangle  | < 1.
\end{equation}

By Lemma \ref{lem21} the quotient $\eta_n (z)= \beta^-_n
(z)/\beta^+_n (z) $ is a well defined analytic function on a
neighborhood of the disc $K_n=\{z: \;|z-z^*_n| \leq \gamma_n \}.$
Moreover, in view of (\ref{g1}) and (\ref{g2}), we have
\begin{equation} \label{a4}
\frac{1}{4c}  \leq |\eta_n (z)| \leq 4c \quad \text{for} \quad n\in
\mathcal{M}^+, \; z \in K_n.
\end{equation}
Since $\eta_n (z)$ does not vanish in $K_n ,$ there is an
appropriate branch $\text{Log}$ of $\log z $ (which depend on $n$)
defined on a neighborhood of $\eta_n (K_n). $  We set
$$
\text{Log} \, (\eta_n (z)) = \log |\eta_n (z)| + i \varphi_n (z);
$$
then
\begin{equation}
\label{a14} \eta_n (z)= \beta^-_n (z)/\beta^+_n (z) =|\eta_n (z)|
e^{i \varphi_n (z)},
\end{equation}
so the square root $\sqrt{\beta^-_n (z)/\beta^+_n (z)}$ is a well
defined analytic function on a neighborhood of $K_n$ by
\begin{equation}
\label{a15} \sqrt{\beta^-_n (z)/\beta^+_n (z)}= \sqrt{|\eta_n
(z)|}e^{\frac{i}{2} \varphi_n (z)}.
\end{equation}

Now the basic equation (\ref{be}) splits into the following two
equations
\begin{eqnarray}
\label{a17} z=\zeta_n^+ (z):=  \alpha_n (z) + \beta^+_n (z)
 \sqrt{\beta^-_n (z)/\beta^+_n (z)}, \\
\label{a18} z=\zeta_n^- (z):=  \alpha_n (z) - \beta^+_n (z)
\sqrt{\beta^-_n (z)/\beta^+_n (z)}.
\end{eqnarray}
 For large enough $|n|,$ each of the equations (\ref{a17})
and (\ref{a18}) has exactly one root in the disc $K_n.$  Indeed, in
view of (\ref{1.37}),
$$
\sup_{|z|\leq 1/2} \left | d \zeta_n^{\pm}/dz \right | \to 0 \quad
\text{as} \quad n \to \infty.
$$
Therefore, for large enough $|n|$ each of the functions $\zeta_n^\pm
$ is a contraction on the disc $K_n, $  which implies that each of
the equations (\ref{a17}) and (\ref{a18}) has at most one root in
the disc $K_n.$ On the other hand,  Lemma~\ref{loc} implies that for
large enough $|n|$ the basic equation (\ref{be}) has exactly two
simple roots in $K_n, $ so each of the equations (\ref{a17}) and
(\ref{a18}) has exactly one root in the disc $K_n.$

For large enough $|n|,$ let $z_1 (n) $ (respectively $z_2 (n) $) be
the only root of the equation (\ref{a17}) (respectively (\ref{a18}))
in the  disc $K_n.$ Of course, we have
 $$ \text{either} \;\;  (i) \;\; z_1 (n) = \lambda_n^-
-n,  \; z_2 (n) = \lambda_n^+ -n  \quad   \text{or} \; \; (ii) \;\;
z_1 (n) = \lambda_n^+ -n, \; z_2 (n) = \lambda_n^- -n.$$ Further we
assume that (i) takes place; the case (ii) may be treated in the
same way, and in both cases we have
\begin{equation}
\label{a24} |z_1 (n) - z_2 (n)| = \gamma_n = |\lambda_n^+
 -  \lambda_n^-|.
\end{equation}

We set
\begin{equation}
\label{c24} f^0(n)= P_n^0 f(n), \quad  g^0(n)= P_n^0 g(n).
\end{equation}
From (\ref{r12}) it follows that $\|P_n - P_n^0\| \to 0.$ Therefore,
$$
\|f(n) - f^0(n) \| = \|(P_n - P_n^0)f(n) \| \leq \|P_n - P_n^0\| \to
0, \quad \|g(n) -g^0 (n)\| \to 0,
$$
so $ |\langle f(n)-f^0(n), g(n)-g^0 (n) \rangle| \to 0.$  Since $
\|f (n)\|^2 =\|f^0 (n)\|^2+ \|f(n) - f^0(n) \|^2 $ and $ \langle
f(n), g(n) \rangle = \langle f^0(n), g^0(n) \rangle + \langle
f(n)-f^0(n), g(n)-g^0 (n) \rangle, $  we obtain
\begin{equation}
\label{a23} \|f^0(n)\|, \,\|g^0(n)\|\to 1,  \quad \limsup_{n\to
\infty} |\langle f(n), g(n) \rangle |=\limsup_{n\to \infty} |\langle
f^0(n), g^0(n) \rangle |.
\end{equation}

By Lemma \ref{lem1},   $f^0 (n)$ is an eigenvector of the matrix
$\begin{pmatrix} \alpha_n (z_1) & \beta_n^- (z_1)
\\ \beta_n^+ (z_1) & \alpha_n (z_1)
\end{pmatrix}$ corresponding to its eigenvalue $z_1=z_1 (n), $ i.e.,
$$ \begin{pmatrix} \alpha_n (z_1)-z_1 &  \beta_n^- (z_1) \\ \beta_n^+ (z_1) &
\alpha_n (z_1)-z_1  \end{pmatrix} f^0 (n) = 0. $$ Therefore, $f^0(n)
$ is proportional to the vector $\left (\frac{z_1 - \alpha_n
(z_1)}{\beta_n^+ (z_1)}, \,1 \right )^T.$ Taking into account
(\ref{a14}), (\ref{a15}) and (\ref{a17}) we obtain
\begin{equation}
\label{a25} f^0 (n) = \frac{\|f^0 (n)\|} {\sqrt{1+ |\eta_n (z_1)|}}
\begin{pmatrix}  \sqrt{|\eta_n (z_1)|}e^{\frac{i}{2}\varphi(z_1)}\\1
\end{pmatrix}.
\end{equation}
In an analogous way, from (\ref{a14}), (\ref{a15}) and (\ref{a18})
it follows
\begin{equation}
\label{a26} g^0 (n) = \frac{\|g^0 (n)\|} {\sqrt{1+ |\eta_n (z_2)|}}
\begin{pmatrix}  -\sqrt{|\eta_n (z_2)|}e^{\frac{i}{2}\varphi(z_2)}\\1
\end{pmatrix}.
\end{equation}
Now,  (\ref{a25}) and (\ref{a26}) imply
\begin{equation}
\label{a27} \langle f^0 (n), g^0 (n) \rangle  =\|f^0 (n)\|\|g^0
(n)\| \frac {1-\sqrt{|\eta_n (z_1)|}\sqrt{ |\eta_n (z_2)|} \,
e^{i\psi_n }} {\sqrt{1+ |\eta_n (z_1)|}\sqrt{1+ |\eta_n (z_2)|}},
\end{equation}
where
$$\psi_n = \frac{1}{2}[\varphi_n (z_1(n)) -\varphi_n (z_2(n)].
$$

Next we explain that
\begin{equation}
\label{a31} \psi_n \to 0 \quad \text{as} \;\; n \to \infty.
\end{equation}
Since $\varphi_n = Im \, \left ( \text{Log}
 \,\eta_n  \right )  $ we obtain, taking into account (\ref{a24}),
 that
$$
|\varphi_n (z_1(n)) -\varphi_n (z_2(n)| \leq \sup_{[z_1,z_2]} \left
| \frac{d}{dz} \left (\text{Log} \,\eta_n \right )  \right | \cdot
\gamma_n,
$$
where $[z_1, z_2]$ denotes the segment with end points $z_1 = z_1
(n)$ and $z_2 = z_2 (n).$

By (\ref{1.37}) in Proposition \ref{bprop} and (\ref{g2}) in Lemma
\ref{lem21} we estimate
$$
\frac{d}{dz} \left (\text{Log} \,\eta_n \right )= \frac{1}{\beta^-_n
(z)} \frac{d\beta^-_n}{dz} (z) -\frac{1}{\beta^+_n (z)}
\frac{d\beta^+_n}{dz} (z), \quad z \in [z_1,z_2],
$$
as follows:
$$
\left | \frac{d}{dz} \left (\text{Log} \,\eta_n \right )  \right |
\leq \frac{\varepsilon_n}{|\beta^-_n (z_n^*)|}
+\frac{\varepsilon_n}{|\beta^+_n (z_n^*)|}
$$
where $\varepsilon_n= C \left ( \mathcal{E}_{|n|} (r)
+\frac{1}{\sqrt{|n|}} \right )  \to 0 \quad \text{as}\; n \to
\infty.$ Therefore,  (\ref{g1}) and (\ref{g4}) imply that $
|\varphi_n (z_1(n)) -\varphi_n (z_2(n)| \leq 4(1+c)\cdot
\varepsilon_n \to 0, $ i.e., (\ref{a31}) holds.

From (\ref{a27}) it follows
\begin{equation}
\label{a32} |\langle f^0 (n), g^0 (n) \rangle |^2=\|f^0 (n)\|^2
\|g^0 (n)\|^2 \cdot \Pi_n ,
\end{equation}
with
\begin{equation}
\label{a33} \Pi_n = \frac{1+|\eta_n (z_1)| |\eta_n
(z_2)|-2\sqrt{|\eta_n (z_1)||\eta_n (z_2)|} \cos \psi_n}{\left (1+
|\eta_n (z_1)| \right )\left (1+ |\eta_n (z_2)| \right )}.
\end{equation}

 Now  (\ref{a31}) implies  $\cos \psi_n
>0 $ for large enough $n,$  so taking into account that $\|f^0 (n)\|,
\|g^0 (n)\|\leq 1, $ we obtain by (\ref{a4})
$$ |\langle f^0 (n), g^0 (n) \rangle
|^2 \leq \Pi_n \leq
 \frac{1+|\eta_n (z_1)|
|\eta_n (z_2)| }{\left (1+ |\eta_n (z_1)|
 \right ) \left (1+ |\eta_n (z_2)|
 \right )} \leq \delta < 1
$$
with
$$
\delta= \sup \left \{\frac{1+xy}{(1+x)(1+y)}: \; \frac{1}{4c} \leq
x,y \leq 4c \right \}.
$$
 Finally, (\ref{a23}) shows that (\ref{a3}) holds,
 which completes the proof of (A).
 \bigskip

 (B)  For every Dirac potential $v$ we set
  \begin{equation}
\label{t1} t_n (z) = \begin{cases}  |\beta^-_n (z)/\beta^+_n (z)|  &
\text{if} \quad \beta^+_n (z) \neq 0,\\
\infty &  \text{if} \quad \beta^+_n (z) = 0, \;\beta^-_n (z) \neq
0,\\
1 &  \text{if} \quad \beta^+_n (z) = 0, \;\beta^-_n (z)=0;
\end{cases}
\end{equation}
then $t_n (z), \; |z|<1, $  is well-defined for large enough $|n|.$

If $v \not \in X^+,$ then there is a subsequence of indices $(n_k) $
in $\mathcal {M}^+$ such that one of the following holds:
\begin{eqnarray}
\label{t2}  t_{n_k} (z_{n_k}^*) \to 0 \quad \text{as} \;\; k \to \infty,\\
\label{t3} t_{n_k} (z_{n_k}^*) \to \infty \quad \text{as} \;\; k \to
\infty.
\end{eqnarray}
Next we consider only the case (\ref{t2}) because the case
(\ref{t3}) could be handled in a similar way  --  if $1/t_{n_k}
(z_{n_k}^*) \to 0, $ then one may exchange the roles of $\beta_n^+$
and $\beta_n^-$ and use the same argument.

 In the above notations, if (\ref{t2}) holds then
there is a sequence $(\tau_k)$ of positive numbers such that
\begin{equation}
\label{t5} t_{n_k} (z) \leq  \tau_k \to 0 \quad  \forall \, z \in
[z_{n_k}^-,z_{n_k}^+],
\end{equation}
where $[z_n^-,z_n^+]$ denotes the segment with end points $z_n^- $
and $z_n^+.$

Indeed, Lemma \ref{lem11} and (\ref{t2}) imply that for large enough
$k$
\begin{equation}
\label{7.23} |\gamma_{n_k}| \leq 2 (|\beta_{n_k}^- (z_{n_k}^*)| +
|\beta_{n_k}^+ (z_{n_k}^*)|) \leq 4 |\beta_{n_k}^+ (z_{n_k}^*)|.
\end{equation}

In view of (\ref{1.37}) in Proposition \ref{bprop}, for  $z \in
[z_n^-,z_n^+]$ and $n \in\mathcal{M}$ with large enough $|n|$ we
have
\begin{equation}
\label{7.24} |\beta_n^\pm (z) - \beta_n^\pm (z_n^*)| \leq
\sup_{[z_n^-,z_n^+]} \left |\frac{\partial \beta_n^\pm }{\partial z}
(z) \right | \cdot |z-z_n^*| \leq \varepsilon_n |\gamma_n|,
\end{equation}
with $\varepsilon_n  \to 0 $  as $|n|\to \infty.$ Therefore, from
(\ref{7.23}) and (\ref{7.24}) it follows that
\begin{equation}
\label{7.26} |\beta_{n_k}^+ (z)| \geq |\beta_{n_k}^+ (z_{n_k}^*)| -
4 \varepsilon_{n_k} |\beta_{n_k}^+ (z_{n_k}^*) | = (1-4
\varepsilon_{n_k}) |\beta_{n_k}^+  (z_{n_k}^*)|.
\end{equation}

On the other hand, (\ref{7.23}) and (\ref{7.24}) imply that
$$
|\beta_{n_k}^- (z)| \leq |\beta_{n_k}^- (z) - \beta_{n_k}^-
(z_{n_k}^*)| + |\beta_{n_k}^- (z_{n_k}^*)| \leq 4\varepsilon_{n_k}
|\beta_{n_k}^+  (z_{n_k}^*)| + |\beta_{n_k}^- (z_{n_k}^*)|.
$$
Thus, since $\varepsilon_{n_k} \to 0,$ we obtain
$$
\frac{|\beta_{n_k}^-  (z)|}{|\beta_{n_k}^+ (z)|} \leq
\frac{4\varepsilon_{n_k} |\beta_{n_k}^+  (z_{n_k}^*)| +
|\beta_{n_k}^- (z_{n_k}^*)|}{(1-4\varepsilon_{n_k})|\beta_{n_k}^+
(z_{n_k}^*)|} = \frac{4\varepsilon_{n_k} +
t_{n_k}(z_{n_k}^*)}{1-4\varepsilon_{n_k}} \to 0,
$$
i. e., (\ref{t5}) holds with $\tau_k =\frac{4\varepsilon_{n_k} +
t_{n_k}(z_{n_k}^*)}{1-4\varepsilon_{n_k}}.  $

Let the vectors $f(n_k), \, g(n_k) \in Ran (P_{n_k}) $ be chosen as
in (\ref{a2a}).  Then $f(n_k) $ and $g(n_k)$ are unit eigenvectors
which corresponds to the simple eigenvalues $\lambda^+_{n_k} $ and
$\lambda^-_{n_k}, $ so they are uniquely determined up to constant
multipliers of absolute value one. Therefore, if the system of root
functions of $L_{Per^+}(v)$ contains Riesz bases, then the system
$\{f(n_k), \, g(n_k): k\in \mathbb{N}\}$ has to be a Riesz basis in
its closed linear span which coincides with the closed linear span
of  $\{Ran \, P_{n_k}, \; k \in \mathbb{N}\}.$ By Lemma~\ref{lemr}
and (\ref{a23}), this would imply
\begin{equation}
\label{t10} \sup_k  \langle  f(n_k),  g(n_k)  \rangle =\sup_k
\langle  f^0 (n_k),  g^0 (n_k)  \rangle <1.
\end{equation}
Thus, the proof of (B) will be completed if we show that (\ref{t10})
fails.

By Lemma \ref{lem1},   $f^0 (n_k)$ is an eigenvector of the matrix
$\begin{pmatrix} \alpha_{n_k} (z^+_{n_k}) & \beta_{n_k}^-
(z^+_{n_k})
\\ \beta_{n_k}^+ (z^+_{n_k}) & \alpha_{n_k} (z^+_{n_k})
\end{pmatrix}$
corresponding to its eigenvalue $z^+_{n_k}, $ so it follows that
$f^0(n) $ is proportional to the vector $\begin{pmatrix}  a(k) \\
1  \end{pmatrix} $ with  $a(k)=\frac{z^+_{n_k} - \alpha_{n_k}
(z^+_{n_k})}
 {\beta_{n_k}^+ (z^+_{n_k})}.$  Moreover,
from (\ref{be}), (\ref{t1}) and (\ref{t5}) it
 follows that
$$
|a(k)| = \sqrt{t_{n_k}(z^+_{n_k})} \leq \sqrt{\tau_k} \to 0 \quad
\text{as} \;\; k \to \infty.
$$
Therefore, we obtain
\begin{equation}
\label{t31} f^0(n_k) = \frac{\|f^0(n_k)\|}{\sqrt{|a(k)|^2+1} }
\begin{pmatrix}  a(k) \\   1  \end{pmatrix} \to
\begin{pmatrix}  0 \\   1  \end{pmatrix}
\quad \text{as} \;\; k \to \infty.
\end{equation}
In the same we obtain that $g^0(n_k) \to
\begin{pmatrix}  0 \\   1  \end{pmatrix}$
as $ k \to \infty.  $ Hence,  $\langle  f^0(n_k),g^0(n_k) \rangle
\to 1$ as  $ k \to \infty,  $  so (\ref{t10} fails, which completes
the proof.

\end{proof}

By Theorem \ref{thm0},  the condition (\ref{g1})  guarantees that
there exists a Riesz basis in $L^2 ([0,\pi], \mathbb{C}^2)$ which
consists of root functions of the operator $L_{Per^\pm} (v).$
Besides the case $v \in X_t $ (see the next section for a definition
of the class of potentials $X_t$)  it seems difficult to verify the
condition (\ref{g1}). Moreover, since the points $z_n^*$ are not
known in advance, in order to check (\ref{g1}) one has to compare
the values of  $\beta^\pm_n (z) $ for all $z$ close to 0.  Next we
give a modification of Theorem~\ref{thm0}, which is more suitable
for applications.

Consider potentials $v$ such that for $n \in \Gamma^+ =2\mathbb{Z}$
(or $n \in \Gamma^- =2\mathbb{Z}+1$) with large enough $|n|$
\begin{equation}
\label{c2}
 \beta_n^- (0)\neq 0, \quad \beta_n^+ (0)\neq 0
 \end{equation}
and
\begin{equation}
\label{c3} \exists d>0 : \;\;  d^{-1}|\beta_n^\pm (0)| \leq
|\beta_n^\pm (z)| \leq d \, |\beta_n^\pm (0)| \quad   \forall z\in
D=\{z: |z| < 1/4 \}.
\end{equation}

\begin{theorem}
\label{thm00}  Suppose $bc =Per^+ $ (or $bc=Per^-$), and $v$ is a
Dirac potential such that (\ref{c2}) and (\ref{c3}) hold for $n \in
\Gamma^+ $  (respectively $n \in \Gamma^- $). Then

 (a) the system of root functions of $L_{Per^+} (v)$
 (respectively $L_{Per^-} (v)$) is
  complete and contains at most finitely many
  linearly independent associated functions;

   (b)   the system of root functions of $L_{Per^+} (v)$
 (respectively $L_{Per^-} (v)$) contains Riesz bases
   if and only if
\begin{equation}
\label{c4} 0 < \liminf_{n \in \Gamma^+ }
 \frac{|\beta_n^- (0)|}{|\beta_n^+ (0)|},
\quad  \limsup_{n \in \Gamma^+ }  \frac{|\beta_n^- (0)|}{|\beta_n^+
(0)|}<\infty
\end{equation}
(or, respectively, $\liminf$ and  $\limsup$ are taken over
$\Gamma^-). $
\end{theorem}

{\em Remark.}    Although the conditions (\ref{c2})--(\ref{c4}) look
too technical there is -- after \cite{DM15, DM10}  -- a well
elaborated technique to evaluate these parameters and check these
conditions. To compare with the case of Hill operators with
trigonometric polynomial coefficients -- see \cite{DM25a,DM25}.

\begin{proof}
By Proposition \ref{bprop}, for large enough $|n|$ the basic
equation
\begin{equation}
\label{c5} (z-\alpha_n (z))^2 = \beta^+_n (z) \beta^-_n (z),
\end{equation}
has exactly two roots (counted with multiplicity) in the disc
$D=\{z: |z|< 1/4\}. $ Therefore, a number $\lambda = n +z $ with $ z
\in D $ is a periodic or antiperiodic eigenvalue of algebraic
multiplicity two if and only if $z \in D$ satisfies the system of
two equations (\ref{c5}) and
\begin{equation}
\label{c6} 2(z-\alpha_n (z)) \frac{d}{dz} \left (z - \alpha_n (z)
\right ) =\frac{d}{dz} \left ( \beta^+_n (z) \beta^-_n (z)\right ).
\end{equation}

In view of \cite[Theorem 9]{DM20}, the system of root functions of
the operator $L_{Per^\pm} (v)$ is complete, so Part (a) of the
theorem will be proved if we show that there are at most finitely
many $n$  such that the system (\ref{c5}), (\ref{c6}) has a solution
$z \in D.$

Suppose that $z^* \in D $ satisfies (\ref{c5}) and (\ref{c6}). By
(\ref{1.37}),  for each $z \in D $
\begin{equation}
\label{c7} \left |\frac{d\alpha_n}{dz}  (z) \right | \leq
\varepsilon_n, \quad  \left |\frac{d\beta^\pm_n}{dz}  (z)\right |
\leq \varepsilon_n \quad \text{with} \;\; \varepsilon_n \to 0
\;\;\text{as} \;\; |n| \to \infty.
\end{equation}
In view of (\ref{c7}), the equation (\ref{c6}) implies
$$
2\left | z^*-\alpha_n (z^*) \right | (1- \varepsilon_n) \leq
\varepsilon_n  \left (|\beta^+_n (z^*)|+|\beta^-_n (z^*)|   \right
).
$$
By (\ref{c5}),
$$
\left | z^*-\alpha_n (z^*) \right |= |\beta_n^+ (z^*) \beta_n^-
(z^*)|^{1/2},
$$
so it follows, in view of (\ref{c3}),
$$
2(1- \varepsilon_n)\leq \varepsilon_n \left ( \left |
\frac{\beta^+_n (z^*)}{\beta^-_n (z^*)} \right |^{1/2}  +\left |
\frac{\beta^-_n (z^*)}{\beta^+_n (z^*)} \right |^{1/2}  \right )
\leq 2 d \varepsilon_n.
$$
Since $\varepsilon_n \to 0 $ as $|n| \to \infty, $  the latter
inequality holds for at most finitely many $n,$ which completes the
proof of (a).

 In view of (a), all but finitely many of the eigenvalues
of $L_{Per^\pm}$ are simple, i.e., $\lambda_n^- \neq \lambda_n^+$
for large enough $|n|.$ One can easily see that Conditions
(\ref{c2})--(\ref{c4}) imply (\ref{g1}), respectively for $n \in
\Gamma^+ $ or $n \in \Gamma^-, $ i.e., $v \in  X^+$  or $v \in X^-.$
Hence (b) follows from Theorem~\ref{thm0}.
\bigskip

{\em Remark.}  For Hill-Schr\"odinger operators with
$L^2$-potentials, an analog of Theorem~\ref{thm00} has been proven
in \cite[Theorem 1]{DM25} (see also \cite[Theorem 2]{DM25a}).

Theorem~\ref{thm0} gives a criterion for existence of Riesz basis
consisting of root functions in the case of Dirac operators
 $L_{Per^\pm} (v) $ with $L^2 $-potentials.
Technically its proof is based on the same argument as in
\cite[Theorem 1]{DM25}. Moreover,  analogs of Theorem~\ref{thm0} and
\ref{thm00} hold for Hill-Schr\"odinger operators with
$H^{-1}$-potentials hold and the proofs are essentially the same.

\end{proof}

\section{Applications}

Consider the classes of Dirac potentials
\begin{equation}
\label{4.2}  X_t =\left \{v=
\begin{pmatrix} 0 & P
\\ Q & 0 \end{pmatrix}, \quad Q (x) = t \overline{P(x)}, \;
 P,Q \in L^2 ([0,\pi]) \right \}, \; t \in \mathbb{R} \setminus \{0\}.
\end{equation}
If $t=1$ we get the class $X_1$ of symmetric Dirac potentials (which
generate self-adjoint Dirac operators);  $X_{-1} $ is the class of
skew-symmetric Dirac potentials.  Next we show that if $v \in X_t $
then the system of root functions of $L_{Per^+}(v) $ or
$L_{Per^-}(v) $ contains Riesz bases.

\begin{proposition}
\label{prop3} Suppose $v \in X_t, \;  t \in \mathbb{R} \setminus
\{0\}.$

 (a)  If $t>0, $ then
there is a symmetric potential $\tilde{v} $ such that $L_{Per^\pm}
(v)$ is similar to the self-adjoint operator $L_{Per^\pm}
(\tilde{v}), $
 so  its spectrum $ Sp  \left
(L_{Per^\pm} (v) \right ) \subset \mathbb{R}. $

(b)  If  $t<0, $ then there is a skew-symmetric potential $\tilde{v}
$ such that $L_{Per^\pm} (v)$ is similar to $L_{Per^\pm}
(\tilde{v}). $
 Moreover, there is an $N=N(v) $ such
that for $|n|>N$ either

(i) $\lambda_n^-$ and $\lambda_n^+$ are simple eigenvalues and
$\overline{\lambda_n^+} =\lambda_n^-, \;  Im \, \lambda^\pm_n \neq
0$

or

 (ii) $\lambda_n^+ =\lambda_n^-$ is a real eigenvalue of
algebraic and geometric multiplicity 2.

(c)  For large enough $|n|$
\begin{equation}
\label{4.25} \overline{\beta^+_n (z_n^*, v)}=t \cdot \beta^-_n
(z_n^*, v),
\end{equation}
which implies $ X_t \subset X^+ \cup X^-. $

(d)   The system of root functions of $L_{Per^+}(v) $ (or
$L_{Per^-}(v) $) contains Riesz bases.

\end{proposition}

\begin{proof}
For every $c\neq 0, $ the Dirac operator $L_{Per^\pm}(v) $ is
similar to the Dirac operator $L_{Per^\pm}(v_c) $ with
$v_c = \begin{pmatrix}   0  & cP  \\
\frac{1}{c} Q   &  0
\end{pmatrix}.$
Indeed, if $C= \begin{pmatrix}   c  & 0  \\
  0  &  1
\end{pmatrix},$
then a simple calculation shows that $ CL_{Per^\pm}(v)=
L_{Per^\pm}(v_c) C. $

If $v \in X_t $ we set $\tilde{v} = v_c  $ with $ c=\sqrt{|t|}.$
Then $ \frac{1}{c} Q = \frac{t}{c} \overline{P}  =
\frac{t}{\sqrt{|t|}}\overline{P}=
 \pm c \overline{P}.
$ Therefore,  $\tilde{v} $  is symmetric or skew-symmetric,
respectively,   for $t>0 $  and  $t<0.$

(b)  By (\ref{adj}),  $ (L_{Per^\pm} (v))^* = L_{Per^\pm} (v^*)$
with
 $$ v^* =\begin{pmatrix} 0 & \overline{Q}
\\ \overline{P} & 0 \end{pmatrix} =
\begin{pmatrix} 0 & tP
\\ \frac{1}{t} Q & 0 \end{pmatrix} = v_t,
$$
so the operator $L_{Per^\pm} (v)$ is similar to its adjoint
operator. Therefore, if $\lambda \in Sp  \left (L_{Per^\pm} (v)
\right ), $ then  $\overline{\lambda} \in Sp  \left (L_{Per^\pm} (v)
\right ) $ as well.

On the other hand by Lemma \ref{loc}, there is an $N=N(v) $ such
that for $|n|>N$ the disc $D_n =\{z: \; |z-n|<1/4\}$ contains
exactly two (counted with algebraic multiplicity) periodic (for even
$n$) or antiperiodic (for odd $n$) eigenvalues of the operator
$L_{Per^\pm}.$  Therefore,
 if
$\lambda \in D_n $ with $  Im \, \lambda \neq 0 $ is an eigenvalue
of $L_{Per^\pm} $  then $\overline{\lambda} \in D_n $ is also an
eigenvalue of $L_{Per^\pm}$ and  $\overline{\lambda} \neq \lambda, $
so $\lambda $ and $\overline{\lambda}$ are simple, i.e., (i) holds.

Suppose $\lambda \in D_n $ is a real eigenvalue. If $\begin{pmatrix}
w_1 \\w_2
\end{pmatrix}$  is a corresponding  eigenvector, then
passing to conjugates we obtain $ L \begin{pmatrix} \overline{w_2}
\\ -\overline{w_1}
\end{pmatrix}
= \lambda  L \begin{pmatrix} \overline{w_2} \\ -\overline{w_1}
\end{pmatrix},
$ i.e., $\begin{pmatrix} \overline{w_2} \\ -\overline{w_1}
\end{pmatrix}$ is also an eigenvector corresponding
to the eigenvalue $\lambda.$ But $\left \langle    \begin{pmatrix}
w_1 \\w_2
\end{pmatrix},
\begin{pmatrix} \overline{w_2} \\ -\overline{w_1}
\end{pmatrix}    \right  \rangle =0, $
so these vector-functions are linearly independent. Hence (ii)
holds.

(c) By (i) and (ii) it follows that
$$
z_n^* = \frac{1}{2} (\lambda_n^- + \lambda_n^+) - n \quad \text{is
real for} \quad |n|>N.
$$
In view of (\ref{beta}), this implies that (\ref{4.25}) holds.

(d) In view of (\ref{4.25}), we have $v \in X, $ so the claim
follows from Theorem~\ref{thm0}.
\end{proof}

 \begin{Example}
{\em If $a, b, A, B $ are non-zero complex numbers
 and
 \begin{equation}
 \label{c15}
v=\begin{pmatrix} 0  & P  \\  Q  & 0   \end{pmatrix} \quad
\text{with} \quad
 P(x)  = a e^{ 2ix} + b e^{- 2ix}, \quad Q(x) = A e^{ 2ix} + B
e^{ -2ix},
\end{equation}
 then the system of root functions of $L_{Per^+}(v)$
 (or $L_{Per^-}(v)$)
 contains at most finitely many linearly independent
 associated functions. Moreover,
the system of root functions of $L_{Per^+}(v) $ contains Riesz bases
always, while the system of root functions of $L_{Per^-} (v)$
contains Riesz bases if and only if $|aA|=|bB|.$}
\end{Example}

Let us mention that if $bc = Per^+$ then it is easy to see by
(\ref{beta}) that $\beta_n^\pm (z)= 0 $ whenever defined, so the
claim follows from Theorem~\ref{thm0}.

If $bc = Per^-,$ then the result follows from Theorem~\ref{thm00}
and the asymptotics
 \begin{equation}
 \label{c16}
\beta_n^+ (0) = A^{\frac{n+1}{2}}a^{\frac{n-1}{2}}4^{-n+1} \left
[\left ( \frac{n-1}{2}  \right )! \right ]^{-2} \left (1+
O(1/\sqrt{|n|} \right ),
 \end{equation}
\begin{equation}
 \label{c17}
\beta_n^- (0) = b^{\frac{n+1}{2}}B^{\frac{n-1}{2}}4^{-n+1} \left
[\left ( \frac{n-1}{2} \right )! \right ]^{-2} \left (1+
O(1/\sqrt{|n|} \right ).
 \end{equation}
 Proofs of (\ref{c16}), (\ref{c17}) and similar asymptotics,
 related to other trigonometric polynomial potentials and implying
 Riesz basis existence or non-existence, will be given elsewhere
 (see similar results for the Hill-Schr\"odinger operator in
 \cite{DM25a,DM25}).

\end{document}